\documentclass{article}
\usepackage{amscd,amsmath}
\usepackage{amssymb}
\usepackage{amsthm}
\usepackage{color}
\usepackage{hyperref}
\usepackage{enumerate}
\usepackage{geometry}
\theoremstyle{definition}
\newtheorem{theorem}{Theorem}[section]
\newtheorem{definition}[theorem]{Definition}
\newtheorem{proposition}[theorem]{Proposition}
\newtheorem{lemma}[theorem]{Lemma}
\newtheorem{corollary}[theorem]{Corollary}
\newtheorem{remark}[theorem]{Remark}

\numberwithin{equation}{section}
\newtheorem{problem}{Problem}

\newtheorem{acknowledgements}{Acknowledgement}

\begin{document}

\title{First Chen Inequality for CR-Warped Product Submanifolds of a Complex Space Form and Applications}

\author{Abdulqader Mustafa\\Department of Mathematics\\ Faculty of Arts and Science, Palestine Technical University, Kadoorei, Tulkarm, Palestine\\abdulqader.mustafa@ptuk.edu.ps\\
\and Monika Sati\\Department of Mathematics \\Govt. PG. College Joshimath, Chamoli Uttarakhand\\monikasati123@gmail.com\\
\and Uday Chand De\\Department of Pure Mathematics\\University of Calcutta, 35, Ballygunge Circular Road, Kolkata-700019, West Bengal, India\\uc de@yahoo.com\\
\and Cenap \"Ozel\\Department of Mathematics\\ Faculty of Science, King Abdulaziz University, 21589, Jeddah, Saudi Arabia\\cozel@kau.edu.sa\\
\and Alexander Pigazzini\\Mathematical and Physical Science Foundation, 4200 Slagelse, Denmark\\pigazzini@topositus.com}
\date{}
\maketitle

\begin{abstract}
In this paper, the first Chen inequality is proved for CR-warped product submanifolds in complex space forms. This inequality involves intrinsic invariants (a leaf-wise $\delta$-invariant and the sectional curvature) controlled by an extrinsic one (the mean curvature vector), which provides an answer to Problem \ref{prob11}. We carefully distinguish the leaf-wise $\delta$-invariant of a factor (used in the bound) from the intrinsic Chen invariant of the same factor, the two being related, on the totally real factor, by the Bishop--O'Neill formula. The bound is sharp and is uniform in the sign of the holomorphic sectional curvature $c$. As a geometric application, we derive necessary conditions for the immersed CR-warped product submanifold to be minimal in a complex space form, providing a partial answer to a well-known problem proposed by S.\,S.\,Chern (Problem \ref{prob3}). For further research directions, we address a couple of open problems (Problem \ref{prob6} and Problem \ref{pqm2}).


\end{abstract}

\sloppy
\section{Introduction}
Warped products play important roles in general relativity, where they provide useful mathematical models of spacetime: the Robertson--Walker spacetime, the Friedmann cosmological models and the standard static spacetime are all warped products. Warped products also provide natural settings to model spacetime near black holes; e.g.\ the Schwarzschild spacetime admits a warped product construction \cite{abd}.

Extrinsic and intrinsic Riemannian invariants have wide applications in many fields of science other than differential geometry; in particular they are of considerable significance in general relativity \cite{2ndi}. Among extrinsic invariants, the shape operator and the squared mean curvature are central; among intrinsic invariants, sectional, Ricci and scalar curvatures together with Chen's $\delta$-invariant are the most prominent. Motivated by the Nash embedding theorem, the program of B.--Y.\ Chen is to search for sharp control of extrinsic quantities by intrinsic ones, and vice versa \cite{aallr4,yyhh88}. This was the original motivation behind the following problem.

\begin{problem}\label{prob11}
\cite{66449d} Establish simple relationships between the main extrinsic invariants and the main intrinsic invariants of a submanifold.
\end{problem}

Several famous results in differential geometry, such as the isoperimetric inequality, Chern--Lashof's inequality and the Gauss--Bonnet theorem, can be regarded as results in this respect.

It is well known \cite{ddyy7} that the following two conditions are necessary for the immersion to be minimal in Euclidean space $\mathbb{E}^m$:

\smallskip
\noindent\textbf{Condition 1:} If $\varphi: M^n\rightarrow \mathbb{E}^m$ is a minimal immersion from a manifold of positive dimension into a Euclidean $m$-space, then $M^n$ is non-compact.

\smallskip
\noindent\textbf{Condition 2:} If $\varphi: M^n\rightarrow \mathbb{E}^m$ is a minimal immersion from a manifold of positive dimension into a Euclidean $m$-space, then the Ricci tensor of $M^n$ is negative semi-definite.

\smallskip
S.\,S.\,Chern asked on page 13 of \cite{CU} to search for further necessary conditions on the Riemannian metric of a submanifold $M^n$ in order to admit an isometric minimal immersion into a Euclidean space. Later, Chen materialized this goal for warped product submanifolds in the following form.

\begin{problem}\label{prob3}
\cite{aallr4} Given a warped product $N_1\times_f N_2$, what are the necessary conditions for the warped product to admit a minimal isometric immersion in a Euclidean $m$-space $\mathbb{E}^m$ (or in $\tilde M^{2m}(c)$)?
\end{problem}

In \cite{ddyy7}, Chen initiated a significant inequality in terms of the intrinsic $\delta$-invariant. This celebrated inequality drew the attention of several authors \cite{yyhh88}. Motivated by Chen's results, we construct a new general inequality in terms of the $\delta$-invariant, this time for CR-warped product submanifolds of complex space forms. As we will see, the natural quantity controlled by the inequality is the \emph{leaf-wise} $\delta$-invariant of the factor (Definition \ref{def:hatdelta}), which for the holomorphic factor $N_T$ coincides with its intrinsic Chen invariant, while for the totally real factor $N_\perp$ is related to its intrinsic Chen invariant by the Bishop--O'Neill formula (Proposition \ref{prop:bishop-oneill-delta}).

We first recall the following algebraic lemma from \cite{ddyy7}.

\begin{lemma}\label{first lemma}
Let $\alpha_1, \alpha_2,\ldots,\alpha_n, \beta$ be $(n+1)$ real numbers ($n\geq 2$) such that
\[
\Bigl(\sum_{i=1}^{n}\alpha_i\Bigr)^{\!2} \;=\; (n-1)\Bigl(\sum_{i=1}^{n}\alpha^2_i+\beta\Bigr).
\]
Then $2\alpha_1\alpha_2\geq\beta$, with equality if and only if $\alpha_1+\alpha_2= \alpha_3=\dots=\alpha_n$.
\end{lemma}

We also recall the definition of the Chen first invariant, which is the main intrinsic invariant in our inequality:
\begin{equation}\label{eq:delta-def}
\delta_{\tilde{M}^{m}}(x)\;=\; \tilde{\tau}(T_{x}\tilde{M}^m)-\inf\bigl\{\tilde{K}(\pi):\pi\subset T_{x}\tilde{M}^{m},\ \dim \pi=2\bigr\},\qquad x\in\tilde M^m.
\end{equation}
The following theorem was first proved for Riemannian submanifolds in real space forms by Chen in \cite{ddyy7}. It is known as the \emph{Chen first inequality}. Since its publication, it has been extended to Riemannian submanifolds in various ambient spaces \cite{55kk99}; however, it had not been proved for CR-warped product submanifolds in any complex ambient manifold. Thus, this work is devoted to proving such an inequality in the setting of CR-warped product submanifolds of complex space forms.

The classical Chen first inequality reads as follows:

\begin{theorem}[Chen 1993, \cite{ddyy7}]\label{thm:chen-original}
Let $M^n$ ($n\geq 3$) be an $n$-dimensional submanifold of a real space form $\tilde{M}^{m}(c)$ of constant sectional curvature $c$. Then
\begin{equation}\label{eq:chen-original}
\delta_M(x) \;\le\; \frac{n^2(n-2)}{2(n-1)}\|\vec H\|^2 + \frac{(n+1)(n-2)}{2}\,c,
\end{equation}
where $\delta_M(x)=\tau(T_xM^n)-\inf K(\pi)$. Equivalently,
\[
\inf K \;\geq\; \tau(T_{x}M^{n})\;-\;\frac{n^{2}(n-2)}{2(n-1)}\|\vec H\|^2\;-\;\frac{(n+1)(n-2)}{2}\,c.
\]
Equality holds if and only if, with respect to a suitable orthonormal frame $e_1,\ldots,e_n,e_{n+1},\ldots,e_m$, the shape operators of $M^n$ in $\tilde M^m(c)$ take the forms
\[
A_{e_{n+1}}=\begin{pmatrix}
\mu_1 & 0 & 0 & \cdots & 0\\ 0 & \mu_2 & 0 &\cdots & 0\\ 0 & 0 & \mu & \cdots & 0\\ \vdots & \vdots & \vdots & \ddots & \vdots\\ 0 & 0 & 0 & 0 & \mu
\end{pmatrix},\quad \mu=\mu_1+\mu_2,
\]
\[
A_{e_{r}}= \begin{pmatrix}
h^{r}_{11}&h^{r}_{12}&0&\cdots&0\\h^{r}_{12}&-h^{r}_{11}&0&\cdots&0\\0&0&0&\cdots&0\\\vdots&\vdots&\vdots&\ddots&\vdots\\0&0&0&0&0
\end{pmatrix}, \quad r=n+2,\ldots,m.
\]
\end{theorem}

The present paper is organized as follows. After this introduction, in Section \ref{sec:prelims} we present preliminaries, basic definitions and formulae, including the precise notion of leaf-wise $\delta$-invariant and its relation to the intrinsic one via Bishop--O'Neill. In Section \ref{sec:lemmas} we prove preparatory lemmas useful for the next section. In Section \ref{sec:main} we consider CR-warped products in complex space forms and prove a general inequality involving the (leaf-wise) $\delta$-invariant and the mean curvature vector, as an answer to Problem \ref{prob11}; the bound holds uniformly in the sign of $c$. In Section \ref{sec:chern} we provide solutions to Chern's problem: necessary conditions for an isometric immersion of a CR-warped product to be minimal (Corollaries \ref{cor:N1} and \ref{cor:N2}), in both the leaf-wise and the intrinsic form. In the last section we address two related open problems.

\section{Preliminaries}\label{sec:prelims}
Let $\tilde{M}^{m}$ be a smooth Riemannian manifold (real dimension $m$). If $X,Y\in T_{x}\tilde{M}^{m}$ are linearly independent, the \emph{sectional curvature} of the $2$-plane $\pi$ spanned by $X$ and $Y$ is
\begin{equation}\label{eq:K-def}
\tilde{K}(X\wedge Y)\;=\; \frac{\tilde{g}(\tilde{R}(X,Y)Y,X)}{\tilde{g}(X,X)\tilde{g}(Y,Y)-(\tilde{g}(X,Y))^{2}},
\end{equation}
where $\tilde{g}$ is the Riemannian metric on $\tilde{M}^{m}$. If $\pi$ is spanned by orthonormal vectors $X,Y$ at $x$, the previous formula reduces to
\begin{equation}\label{eq:K-on}
\tilde{K}(\pi)\;=\; \tilde{K}_{\tilde{M}^{m}}(X\wedge Y)\;=\; \tilde{g}(\tilde{R}(X,Y)Y,X).
\end{equation}
The \emph{scalar curvature} of $\tilde M^m$ at $x$ is
\begin{equation}\label{eq:scalar-def}
\tilde{\tau}(T_{x}\tilde{M}^{m})\;=\; \sum_{1\leq i<j\leq m}\tilde{K}_{ij},
\end{equation}
where $\tilde{K}_{ij}=\tilde{K}(e_{i}\wedge e_{j})$. Equivalently,
\begin{equation}\label{eq:scalar-2tau}
2\tilde{\tau}(T_{x}\tilde{M}^{m})\;=\; \sum_{1\leq i\neq j\leq m}\tilde{K}_{ij}.
\end{equation}

\paragraph{Partial scalar curvature and Chen invariants on subspaces.}
For a linear subspace $V\subseteq T_x\tilde M^m$ of dimension $k\geq 2$, the \emph{partial scalar curvature of $\tilde M^m$ on $V$ at $x$} is
\begin{equation}\label{eq:partial-tau}
\tilde\tau(V)\;:=\;\sum_{1\leq i<j\leq k}\tilde K(e_i\wedge e_j),
\end{equation}
where $\{e_1,\dots,e_k\}$ is any orthonormal basis of $V$ (the sum is independent of the basis). The \emph{leaf-wise Chen first invariant of $\tilde M^m$ on $V$} is
\begin{equation}\label{eq:delta-V}
\delta(V)\;:=\;\tilde\tau(V)\;-\;\inf\bigl\{\tilde K(\pi):\pi\subseteq V,\ \dim\pi=2\bigr\}.
\end{equation}
When $V=T_x\tilde M^m$, $\delta(V)=\delta_{\tilde M^m}(x)$ recovers \eqref{eq:delta-def}.

\paragraph{Gradient and Laplacian.}
For a smooth function $\psi$ on $\tilde M^m$, the gradient $\tilde\nabla\psi$ and Laplacian $\Delta\psi$ are
\begin{equation}\label{eq:grad-lap}
\tilde{g}(\tilde{\nabla}\psi,X)\;=\; X(\psi),\qquad \Delta\psi\;=\; \sum_{i=1}^{m}\bigl((\tilde{\nabla}_{e_{i}}e_{i})\psi - e_{i}e_{i}\psi\bigr),
\end{equation}
i.e.\ the geometer's sign convention ($\Delta = -\mathrm{div}\,\mathrm{grad}$).

\paragraph{Warped products.}
In an attempt to construct manifolds of negative curvature, R.\,L.\,Bishop and B.\,O'Neill \cite{pom} introduced warped product manifolds. Let $N_1,N_2$ be two Riemannian manifolds with metrics $g_{N_1}, g_{N_2}$, and let $f>0$ be a smooth function on $N_1$. The warped product $M:=N_1\times_f N_2$ is the product manifold $N_1\times N_2$ equipped with the metric $g_M = g_{N_1}+f^2 g_{N_2}$. The warped product is \emph{trivial} if $f$ is constant. For a non-trivial warped product, $\mathfrak D_1$ and $\mathfrak D_2$ denote the distributions tangent to the leaves and fibers respectively.

\paragraph{Bishop--O'Neill formulas.}
For $X,Y$ tangent to $N_1$ and $V,W$ tangent to $N_2$, the curvature of $M$ satisfies
\begin{equation}\label{eq:BO-RXY}
R^M(X,Y)Z\;=\;R^{N_1}(X,Y)Z\quad\text{if }X,Y,Z\in TN_1,
\end{equation}
and, for unit $V,W\in TN_2$ (i.e.\ $g_M$-orthonormal in $TN_2$, equivalently $\tilde V:=fV,\ \tilde W:=fW$ are $g_{N_2}$-orthonormal at the corresponding point of $N_2$),
\begin{equation}\label{eq:BO-fiber-K}
K^M(V\wedge W)\;=\;\frac{1}{f^2}\,K^{N_2}(\tilde V\wedge\tilde W)\;-\;\frac{\|\nabla f\|^2}{f^2}.
\end{equation}
The leaves $N_1\times\{q\}$ are totally geodesic in $M$, while the fibers $\{p\}\times N_2$ are totally umbilical in $M$ with mean curvature vector $-\nabla\ln f$.

Let $\{e_1,\ldots,e_{n_1},e_{n_1+1},\ldots,e_n\}$ be a local orthonormal frame of $TM$ with $\{e_1,\ldots,e_{n_1}\}$ tangent to $N_1$ and $\{e_{n_1+1},\ldots,e_n\}$ tangent to $N_2$. The fundamental warped-product identity (see e.g. \cite{aallr4,yyhh88,55kk99}) is
\begin{equation}\label{eq:warp-id}
\sum_{a=1}^{n_1} \sum_{A=n_1+1}^{n} K^M(e_a\wedge e_A)\;=\; \frac{n_2\,\Delta f}{f}.
\end{equation}

\paragraph{Second fundamental form of an isometric immersion.}
For an isometric immersion $\varphi:M^n\to\tilde M^m$ with second fundamental form $h$ and shape operator $A_\xi$,
\begin{equation}\label{eq:shape}
g(A_{\xi}X,Y)=g(h(X,Y),\xi),\qquad X,Y\in\Gamma(TM^n),\ \xi\in\Gamma(T^\perp M^n).
\end{equation}
Choose a local orthonormal frame $\{e_1,\ldots,e_n,e_{n+1},\ldots,e_m\}$ adapted to $\varphi$, with $\{e_1,\ldots,e_n\}$ tangent and $\{e_{n+1},\ldots,e_m\}$ normal to $M^n$. The mean curvature vector is
\begin{equation}\label{eq:H-def}
\vec{H}(x)=\frac{1}{n}\sum_{i=1}^{n} h(e_i,e_i).
\end{equation}
$M^n$ is \emph{minimal} if $\vec H\equiv 0$ and \emph{totally umbilical} if $h(X,Y)=g(X,Y)\vec H$ for all $X,Y$ \cite{2211gg}.

The Gauss equation reads
\begin{equation}\label{eq:gauss}
R(X,Y,Z,W)=\tilde{R}(X,Y,Z,W)+g(h(X,W),h(Y,Z))-g(h(X,Z),h(Y,W)).
\end{equation}
We denote the coefficients of $h$ in the chosen frame by
\begin{equation}\label{eq:hrij}
h^{r}_{ij}=g(h(e_i,e_j),e_r),\qquad i,j\in\{1,\ldots,n\},\ r\in\{n+1,\ldots,m\}.
\end{equation}

From \eqref{eq:hrij}, \eqref{eq:K-on}, \eqref{eq:gauss} we obtain
\begin{equation}\label{eq:K-Gauss}
K(e_i\wedge e_j)=\tilde{K}(e_i\wedge e_j)+\sum_{r=n+1}^{m}\bigl(h^{r}_{ii}h^{r}_{jj}-(h^{r}_{ij})^{2}\bigr).
\end{equation}
Summing over the tangent frame and using \eqref{eq:scalar-2tau},
\begin{equation}\label{eq:gauss-trace}
2\tau(T_x M^n)\;=\;2\tilde{\tau}(T_x M^n)+n^2\|\vec H\|^2-\|h\|^2,
\end{equation}
where $\tilde\tau(T_xM^n)=\sum_{1\leq i<j\leq n}\tilde K(e_i\wedge e_j)$ is the partial scalar curvature of $\tilde M^m$ on $T_xM^n$. More generally, for any subspace $V\subseteq T_xM^n$,
\begin{equation}\label{eq:gauss-trace-V}
2\tau^M(V)\;=\;2\tilde\tau(V)+\sum_r\Bigl[\bigl(\mathrm{tr}_V A^r\bigr)^2-\|A^r|_V\|^2\Bigr],
\end{equation}
where $\tau^M(V)=\sum_{e_i,e_j\in V,\ i<j}K^M(e_i\wedge e_j)$, $A^r=A_{e_r}$, and the trace and squared norm of $A^r|_V$ are taken over an orthonormal basis of $V$.

\paragraph{Complex space forms.}
Let $\tilde M^{2m}$ be a smooth manifold (real dimension $2m$) with almost complex structure $J$ ($J^2=-I$). If $J$ is integrable (the Nijenhuis tensor vanishes), $(\tilde M^{2m},J)$ is a complex manifold. If $\tilde M^{2m}$ also carries a Hermitian metric $\tilde g$ (so $\tilde g(JX,JY)=\tilde g(X,Y)$) and the K\"ahler form $\omega(X,Y)=\tilde g(JX,Y)$ is closed, then $(\tilde M^{2m},J,\tilde g)$ is K\"ahler. Equivalently, $\tilde\nabla J=0$. A K\"ahler manifold of constant holomorphic sectional curvature $c\in\mathbb{R}$ is a \emph{complex space form} $\tilde M^{2m}(c)$. Its Riemannian curvature tensor is
\begin{multline}\label{eq:csf-R}
\tilde{R}(X,Y,Z,W)= \frac{c}{4}\Bigl\{\tilde{g}(X,W)\tilde{g}(Y,Z)-\tilde{g}(X,Z)\tilde{g}(Y,W)\\+\tilde{g}(JX,W)\tilde{g}(JY,Z)-\tilde{g}(JX,Z)\tilde{g}(JY,W)+2\,\tilde{g}(X,JY)\tilde{g}(JZ,W)\Bigr\}.
\end{multline}
From \eqref{eq:csf-R} one computes the sectional curvature of a 2-plane spanned by orthonormal $X,Y$:
\begin{equation}\label{eq:K-csf}
\tilde K(X\wedge Y)\;=\; \frac{c}{4}\bigl(1+3\,\tilde g(JX,Y)^2\bigr).
\end{equation}
In particular, a \emph{holomorphic} 2-plane (i.e.\ $Y=JX$) has $\tilde K=c$, while a 2-plane on which $J$ acts as zero (i.e.\ $JX\perp Y$) has $\tilde K=c/4$. For any 2-plane $\pi$ in $\tilde M^{2m}(c)$,
\begin{equation}\label{eq:K-min-max}
\min\bigl(\tfrac{c}{4},\,c\bigr)\;\leq\;\tilde K(\pi)\;\leq\;\max\bigl(\tfrac{c}{4},\,c\bigr).
\end{equation}
For a $J$-invariant subspace $V$ of even real dimension $\geq 2$, both bounds are attained: the minimum on a non-$J$-invariant 2-plane (which always exists as soon as $\dim V\geq 4$, and for $\dim V=2$ if and only if $V$ is not itself $J$-invariant) and the maximum on a holomorphic 2-plane.

\paragraph{CR submanifolds.}
Following \cite{AKU17}, a submanifold $M^n$ of an almost Hermitian manifold $\tilde M^{2m}$ is called a \emph{CR-submanifold} if $TM^n$ admits an orthogonal decomposition into a holomorphic distribution $\mathfrak D_T$ ($J\mathfrak D_T\subseteq TM^n$) and a totally real distribution $\mathfrak D_\perp$ ($J\mathfrak D_\perp\subseteq T^\perp M^n$):
\begin{enumerate}[\rm (i)]
\item $TM^{n}=\mathfrak{D}_T\oplus\mathfrak{D}_\perp$;
\item $\mathfrak{D}_T$ is $J$-invariant: $J\mathfrak{D}_T=\mathfrak{D}_T$;
\item $\mathfrak{D}_\perp$ is anti-invariant: $J\mathfrak{D}_\perp\subseteq T^\perp M^n$.
\end{enumerate}
Denoting by $\nu$ the maximal $J$-invariant subbundle of $T^\perp M^n$,
\begin{equation}\label{eq:normal-decomp}
T^\perp M^n=F\mathfrak{D}_\perp \oplus \nu.
\end{equation}
A \emph{CR-warped product submanifold} is a warped product $M^n=N_T\times_f N_\perp$, where $N_T$ is holomorphic ($TN_T=\mathfrak D_T$) and $N_\perp$ is totally real ($TN_\perp=\mathfrak D_\perp$). Note that, since $J$ acts as a complex structure on $\mathfrak D_T$, the dimension $n_1=\dim N_T$ is necessarily even.

\paragraph{Ambient scalar curvature of a CR submanifold.}
For a CR submanifold $M^n$ of $\tilde M^{2m}(c)$ with $n_1=\dim\mathfrak D_T$ and $n_2=\dim\mathfrak D_\perp$ ($n=n_1+n_2$), a direct computation from \eqref{eq:K-csf} gives
\begin{equation}\label{eq:tilde-tau}
2\tilde\tau(T_xM^n)\;=\; \frac{c}{4}\bigl[n(n-1)+3n_1\bigr].
\end{equation}
Indeed, choose an orthonormal basis of $T_xM^n$ such that $\{e_1,\ldots,e_{n_1}\}\subset \mathfrak D_T$ with $Je_{2a-1}=e_{2a}$ for $a=1,\ldots,n_1/2$, and $\{e_{n_1+1},\ldots,e_n\}\subset \mathfrak D_\perp$. Then $\tilde g(Je_i,e_j)^2=1$ exactly for the $n_1$ ordered pairs $(2a-1,2a)$ and $(2a,2a-1)$, and vanishes otherwise. Summing \eqref{eq:K-csf} over $i\neq j$ yields \eqref{eq:tilde-tau}.

\begin{remark}\label{rem:tilde-tau-factor}
Equivalently, for the holomorphic factor $T_xN_T$ alone (which is $J$-invariant of dimension $n_1$ and contains $n_1/2$ holomorphic 2-planes from the chosen basis):
\begin{equation}\label{eq:tilde-tau-1}
2\tilde\tau(T_xN_T)\;=\; \frac{c}{4}\bigl[n_1(n_1-1)+3n_1\bigr]\;=\;\frac{c}{4}\,n_1(n_1+2),
\end{equation}
while for the totally real factor $T_xN_\perp$, since $J\mathfrak D_\perp\perp \mathfrak D_\perp$,
\begin{equation}\label{eq:tilde-tau-2}
2\tilde\tau(T_xN_\perp)\;=\; \frac{c}{4}\,n_2(n_2-1).
\end{equation}
\end{remark}

Combining \eqref{eq:gauss-trace} with \eqref{eq:tilde-tau} yields the fundamental identity for a CR submanifold of $\tilde M^{2m}(c)$:
\begin{equation}\label{eq:fundamental}
\;n^2\|\vec H\|^2\;=\;2\tau(T_xM^n)+\|h\|^2-\frac{c}{4}\bigl[n(n-1)+3n_1\bigr].\;
\end{equation}

\paragraph{Leaf-wise $\delta$-invariant of a CR-warped product factor.}
Throughout the rest of the paper, fix a CR-warped product $M^n=N_T\times_f N_\perp$ isometrically immersed in a complex space form $\tilde M^{2m}(c)$. The natural quantity controlled by the inequalities we prove is the following.

\begin{definition}\label{def:hatdelta}
For $x\in M^n$ and $V\in\{T_xN_T,\,T_xN_\perp\}$ with $\dim V\geq 2$, set
\begin{equation}\label{eq:hatdelta}
\hat\delta(V)(x)\;:=\;\tau^M(V)\;-\;\inf\bigl\{K^M(\pi):\pi\subseteq V,\ \dim\pi=2\bigr\}.
\end{equation}
We call $\hat\delta(T_xN_T)$ and $\hat\delta(T_xN_\perp)$ the \emph{leaf-wise first Chen invariants} of $N_T$ and $N_\perp$ at $x$. Throughout, $K^M$ and $\tau^M$ denote sectional and partial scalar curvatures of $M^n$ (\emph{not} of the factors).
\end{definition}

We now relate $\hat\delta(\cdot)$ to the genuinely intrinsic Chen invariants of the factor manifolds.

\begin{proposition}\label{prop:bishop-oneill-delta}
Let $x=(p,q)\in M^n=N_T\times_f N_\perp$. Then:
\begin{enumerate}[\rm(a)]
\item Since the leaf $N_T\times\{q\}$ is totally geodesic in $M$ and the metric induced by $g_M$ on it is $g_{N_T}$,
\begin{equation}\label{eq:identi-NT}
K^M|_{T_xN_T}=K^{N_T}\bigr|_{T_pN_T},\qquad \tau^M(T_xN_T)=\tau^{N_T}(p),\qquad\hat\delta(T_xN_T)(x)=\delta_{N_T}(p).
\end{equation}
\item Since the fiber $\{p\}\times N_\perp$ carries the metric $f(p)^2\,g_{N_\perp}$ in $g_M$, the Bishop--O'Neill formula \eqref{eq:BO-fiber-K} gives, for any orthonormal $V,W\in T_xN_\perp$,
\begin{equation}\label{eq:identi-Nperp-K}
K^M(V\wedge W)\;=\;\frac{1}{f(p)^2}\,K^{N_\perp}(\tilde V\wedge\tilde W)\;-\;\frac{\|\nabla f(p)\|^2}{f(p)^2},
\end{equation}
where $\tilde V=f(p)\,V$, $\tilde W=f(p)\,W$ are $g_{N_\perp}$-orthonormal. Hence
\begin{equation}\label{eq:identi-Nperp-tau}
\tau^M(T_xN_\perp)\;=\;\frac{1}{f(p)^2}\,\tau^{N_\perp}(q)\;-\;\binom{n_2}{2}\frac{\|\nabla f(p)\|^2}{f(p)^2},
\end{equation}
and, taking infima of \eqref{eq:identi-Nperp-K} over 2-planes in $T_xN_\perp$,
\begin{equation}\label{eq:identi-Nperp-delta}
\hat\delta(T_xN_\perp)(x)\;=\;\frac{1}{f(p)^2}\,\delta_{N_\perp}(q)\;-\;\Bigl[\binom{n_2}{2}-1\Bigr]\frac{\|\nabla f(p)\|^2}{f(p)^2}.
\end{equation}
\end{enumerate}
\end{proposition}

\begin{proof}
Part (a) follows from the standard fact that for an isometric, totally geodesic embedding $\iota:L\hookrightarrow(M,g_M)$, the intrinsic curvature of $L$ (in the induced metric) coincides with the restriction of $K^M$ to 2-planes tangent to $L$ (Gauss equation with $h=0$). For the warped product, the leaf $N_T\times\{q\}$ is totally geodesic by Bishop--O'Neill \cite{pom,abd}, and the induced metric on it is $g_{N_T}$.

For (b), \eqref{eq:identi-Nperp-K} is the Bishop--O'Neill formula \eqref{eq:BO-fiber-K} restricted to a 2-plane in $T_xN_\perp$. Summing over a $g_M$-orthonormal basis of $T_xN_\perp$ yields \eqref{eq:identi-Nperp-tau}. Subtracting an infimum (which corresponds, via the bijection $\pi\mapsto\tilde\pi$ between $g_M$-orthonormal and $g_{N_\perp}$-orthonormal 2-planes, to an infimum on $N_\perp$) gives
\[
\inf_{\pi\subset T_xN_\perp}K^M(\pi)\;=\;\frac{1}{f^2}\inf_{\tilde\pi\subset T_qN_\perp}K^{N_\perp}(\tilde\pi)\;-\;\frac{\|\nabla f\|^2}{f^2},
\]
and \eqref{eq:identi-Nperp-delta} follows by subtraction with \eqref{eq:identi-Nperp-tau}.
\end{proof}

\section{Basic Lemmas}\label{sec:lemmas}
This section collects three computational lemmas. The first is the warped-product identity, already stated as \eqref{eq:warp-id}: for $M^n=N_1\times_f N_2$,
\begin{equation}\label{eq:warp-id-bis}
\sum_{a=1}^{n_1}\sum_{A=n_1+1}^{n}K^M(e_{a}\wedge e_A)=\frac{n_2\Delta f}{f}.
\end{equation}

\begin{lemma}\label{Second lemma}
Let $\varphi$ be an isometric immersion of a warped product $M^{n}=N_1\times_f N_2$ into a Riemannian manifold $\tilde{M}^{m}$. Then
\begin{multline}\label{eq:lemma2}
\tfrac{1}{2}\!\!\sum_{\substack {i,j=1\\i\neq j}}^{n}\!(h_{ij}^{n+1})^2+\tfrac{1}{2}\sum_{r=n+2}^{m}\sum_{i,j=1}^{n}(h_{ij}^{r})^2+\sum_{r=n+2}^{m} h_{11}^{r}h_{22}^{r}-\!\!\sum_{r=n+1}^{m}\!\!(h_{12}^{r})^2\\
=\;\tfrac{1}{2}\!\!\sum_{\substack {i,j=3\\i\neq j}}^{n}\!(h_{ij}^{n+1})^2+\tfrac{1}{2}\sum_{r=n+2}^{m}\sum_{i,j=3}^{n}(h_{ij}^{r})^2+ \tfrac{1}{2}\sum_{r=n+2}^{m}(h_{11}^{r}+h_{22}^{r})^{2}\\+\sum_{r=n+1}^{m}\sum_{j=3}^{n}\bigl((h_{1j}^{r})^2+(h_{2j}^{r})^2\bigr).
\end{multline}
\end{lemma}

\begin{proof}
Split the index sets $\{1,\ldots,n\}=\{1,2\}\cup\{3,\ldots,n\}$. The first two summands on the left expand as
\begin{align*}
\tfrac{1}{2}\!\!\sum_{\substack {i,j=1\\i\neq j}}^{n}\!(h_{ij}^{n+1})^2&=\tfrac{1}{2}\!\!\sum_{\substack {i,j=3\\i\neq j}}^{n}\!(h_{ij}^{n+1})^2+(h_{12}^{n+1})^2+\sum_{j=3}^{n}\bigl((h_{1j}^{n+1})^2+(h_{2j}^{n+1})^2\bigr),\\
\tfrac{1}{2}\sum_{r=n+2}^{m}\sum_{i,j=1}^{n}(h_{ij}^{r})^{2}&=\tfrac{1}{2}\sum_{r=n+2}^{m}\sum_{i,j=3}^{n}(h_{ij}^{r})^{2}+\sum_{r=n+2}^{m}\sum_{j=3}^{n}\bigl((h_{1j}^{r})^2+(h_{2j}^r)^2\bigr)\\
&\quad + \sum_{r=n+2}^{m}(h_{12}^{r})^{2}+\tfrac{1}{2}\sum_{r=n+2}^{m}\bigl((h_{11}^{r})^2+(h_{22}^{r})^{2}\bigr).
\end{align*}
Using
\[
\sum_{r=n+2}^{m}h_{11}^{r}h_{22}^{r}+\tfrac{1}{2}\!\!\sum_{r=n+2}^{m}\!\!\bigl((h_{11}^r)^{2}+ (h_{22}^{r})^{2}\bigr)=\tfrac{1}{2}\!\!\sum_{r=n+2}^{m}\!\!\bigl(h_{11}^r+h_{22}^{r}\bigr)^{2},
\]
the identity follows by direct substitution; the term $(h_{12}^{n+1})^2+\sum_{r=n+2}^m(h_{12}^r)^2-\sum_{r=n+1}^m(h_{12}^r)^2=0$ cancels.
\end{proof}

\begin{lemma}\label{Third lemma}
Let $\varphi$ be an isometric immersion of a warped product $M^n=N_1\times_f N_2$ into a Riemannian manifold $\tilde{M}^m$. Then
\begin{multline}\label{eq:lemma3}
\tfrac{1}{2}\!\!\sum_{\substack {i,j=3\\i\neq j}}^{n}\!(h_{ij}^{n+1})^2+\tfrac{1}{2}\sum_{r=n+2}^{m}\sum_{i,j=3}^{n}(h_{ij}^r)^2+\sum_{r=n+1}^{m}\sum_{j=3}^{n}\bigl((h_{1j}^{r})^2+(h_{2j}^r)^2\bigr)\\
=\tfrac{1}{2}\!\!\sum_{\substack {a,b=3\\a\neq b}}^{n_1}\!(h_{ab}^{n+1})^2+\tfrac{1}{2}\!\!\!\sum_{\substack {A,B=n_1+1\\A\neq B}}^{n}\!\!\!(h_{AB}^{n+1})^2+\tfrac{1}{2}\sum_{r=n+2}^{m}\sum_{a,b=3}^{n_1}(h_{ab}^r)^2\\
+\tfrac{1}{2}\sum_{r=n+2}^{m}\sum_{A,B=n_1+1}^{n}(h_{AB}^r)^2+\sum_{r=n+1}^{m}\sum_{a=3}^{n_1}\bigl((h_{1a}^r)^2+(h_{2a}^r)^2\bigr)\\
+\sum_{r=n+1}^{m}\sum_{a=3}^{n_1}\sum_{A=n_1+1}^{n}(h_{aA}^r)^2.
\end{multline}
\end{lemma}

\begin{proof}
Split $\{3,\ldots,n\}=\{3,\ldots,n_1\}\cup\{n_1+1,\ldots,n\}$, so that
\begin{align*}
\tfrac{1}{2}\!\!\sum_{\substack {i,j=3\\i\neq j}}^{n}\!(h_{ij}^{n+1})^2&=\tfrac{1}{2}\!\!\sum_{\substack {a,b=3\\a\neq b}}^{n_1}\!(h_{ab}^{n+1})^2+\tfrac{1}{2}\!\!\sum_{\substack {A,B=n_1+1\\A\neq B}}^{n}\!\!(h_{AB}^{n+1})^2+\!\!\sum_{a=3}^{n_1}\sum_{A=n_1+1}^{n}\!\!(h_{aA}^{n+1})^2,\\
\tfrac{1}{2}\sum_{r=n+2}^{m}\sum_{i,j=3}^{n}(h_{ij}^r)^2&=\tfrac{1}{2}\sum_{r=n+2}^{m}\sum_{a,b=3}^{n_1}(h_{ab}^r)^2+\tfrac{1}{2}\sum_{r=n+2}^{m}\sum_{A,B=n_1+1}^{n}(h_{AB}^r)^2\\
&\quad+\sum_{r=n+2}^{m}\sum_{a=3}^{n_1}\sum_{A=n_1+1}^{n}(h_{aA}^r)^2,\\
\sum_{r=n+1}^{m}\sum_{j=3}^{n}\bigl((h_{1j}^r)^2+(h_{2j}^r)^2\bigr)&=\sum_{r=n+1}^{m}\sum_{a=3}^{n_1}\bigl((h_{1a}^r)^2+(h_{2a}^r)^2\bigr)\\
&\quad+\sum_{r=n+1}^{m}\sum_{A=n_1+1}^{n}\bigl((h_{1A}^r)^2+(h_{2A}^r)^2\bigr).
\end{align*}
Combining and using the identity
\[
\sum_{a=3}^{n_1}\!\sum_{A=n_1+1}^{n}\!(h_{aA}^{n+1})^2+\!\!\sum_{r=n+2}^{m}\sum_{a=3}^{n_1}\sum_{A=n_1+1}^{n}\!\!(h_{aA}^r)^2\;=\;\sum_{r=n+1}^{m}\sum_{a=3}^{n_1}\sum_{A=n_1+1}^{n}(h_{aA}^r)^2
\]
yields \eqref{eq:lemma3}.
\end{proof}

\section{First Chen Inequality for CR-Warped Product Submanifolds}\label{sec:main}

For the statement of the main theorem we abbreviate
\begin{equation}\label{eq:Kmin-def}
\tilde K_{\min}(V):=\inf\bigl\{\tilde K(\pi):\pi\subseteq V,\ \dim\pi=2\bigr\}
\end{equation}
for a subspace $V\subseteq T_x\tilde M^{2m}(c)$ of real dimension $\geq 2$. By \eqref{eq:K-min-max},
\begin{equation}\label{eq:Kmin-csf}
\tilde K_{\min}(V)=\min\bigl(\tfrac c4,\,c\bigr)=\begin{cases}c/4&\text{if }c\geq 0,\\ c&\text{if }c\leq 0,\end{cases}
\end{equation}
whenever $V$ is $J$-invariant of even real dimension $\geq 2$; whereas $\tilde K_{\min}(V)=c/4$ whenever $V$ is totally real (since then no holomorphic 2-plane lies in $V$). In particular
\begin{equation}\label{eq:Kmin-cases}
\tilde K_{\min}(T_xN_T)\;=\;\min\bigl(\tfrac c4,\,c\bigr),\qquad \tilde K_{\min}(T_xN_\perp)\;=\;\tfrac c4.
\end{equation}

\begin{theorem}\label{maintheorem}
Let $\varphi : M^n=N_T\times_{f} N_{\perp}\to \tilde{M}^{2m}(c)$ be an isometric immersion of a CR-warped product submanifold into a complex space form of constant holomorphic sectional curvature $c$. Let $n_1=\dim N_T$ (necessarily even) and $n_2=\dim N_\perp$, with $n=n_1+n_2$. Then for each $x\in M^n$ the following inequalities hold:
\begin{enumerate}[\rm (i)]
\item If $n_1\geq 2$, then 
\begin{equation}\label{eq:main-i}
\;\hat\delta(T_xN_T)(x)\;\leq\;\frac{n^2}{2}\|\vec H\|^2-\frac{n_2\Delta f}{f}+\frac{n_1(n_1+2n_2+2)}{2}\cdot\frac{c}{4}-\tilde K_{\min}(T_xN_T).\;
\end{equation}
\item If $n_2\geq 2$, then
\begin{equation}\label{eq:main-ii}
\;\hat\delta(T_xN_\perp)(x)\;\leq\;\frac{n^2}{2}\|\vec H\|^2-\frac{n_2\Delta f}{f}+\frac{n_2(n_2+2n_1-1)}{2}\cdot\frac{c}{4}-\frac{c}{4}.\;
\end{equation}
\end{enumerate}
Equality holds at $x$ in \eqref{eq:main-i} (resp.\ \eqref{eq:main-ii}) if and only if both of the following are satisfied:
\begin{enumerate}[\rm (E1)]
\item the infimum defining $\hat\delta(T_xN_T)$ (resp.\ $\hat\delta(T_xN_\perp)$) is attained on a 2-plane $\pi_*$ with $\tilde K(\pi_*)=\tilde K_{\min}(T_xN_T)$ (resp.\ $\tilde K(\pi_*)=c/4$, which is automatic);
\item there exist orthonormal bases of $T_xM^n$ and $T_x^\perp M^n$ in which the shape operators take the canonical forms detailed in Theorem \ref{thm:equality} below.
\end{enumerate}
In every such equality case, the CR-warped product is \emph{mixed totally geodesic} in $\tilde M^{2m}(c)$, both $\mathfrak D_T$-minimal and $\mathfrak D_\perp$-minimal, and consequently minimal in $\tilde M^{2m}(c)$ at $x$.
\end{theorem}

\begin{proof}
Fix $x\in M^n$. Choose an orthonormal basis $\{e_1,\ldots,e_{n_1},e_{n_1+1},\ldots,e_n\}$ of $T_xM^n$ with $\{e_1,\ldots,e_{n_1}\}\subset T_xN_T$ and $\{e_{n_1+1},\ldots,e_n\}\subset T_xN_\perp$, and an orthonormal basis $\{e_{n+1},\ldots,e_{2m}\}$ of $T_x^\perp M^n$ with $e_{n+1}$ in the direction of $\vec H$ (if $\vec H(x)\neq 0$; otherwise $e_{n+1}$ is arbitrary). Throughout this proof, indices $a,b$ run in $\{1,\dots,n_1\}$, $A,B$ in $\{n_1+1,\dots,n\}$, and $r,s$ in $\{n+1,\dots,2m\}$.

\medskip
\noindent\textbf{Proof of (i).}\ \ Let $\pi_1=\mathrm{Span}\{e_1,e_2\}\subseteq T_xN_T$ be \emph{any} 2-plane (we do \emph{not} assume $\pi_1$ is non-$J$-invariant). Denote $\tilde K_1:=\tilde K(\pi_1)$. The fundamental identity \eqref{eq:fundamental} reads
\begin{equation}\label{eq:F1}
n^2\|\vec H\|^2 \;=\; 2\tau(T_xM^n)+\|h\|^2-\frac{c}{4}\bigl[n(n-1)+3n_1\bigr].
\end{equation}
Using $\bigl(\sum_a h^{n+1}_{aa}\bigr)^2+\bigl(\sum_A h^{n+1}_{AA}\bigr)^2+2\sum_{a,A}h^{n+1}_{aa}h^{n+1}_{AA}=\bigl(\sum_i h^{n+1}_{ii}\bigr)^2=n^2\|\vec H\|^2$, \eqref{eq:F1} is equivalent to
\begin{multline}\label{eq:F2}
\Bigl(\sum_{a=1}^{n_1}h_{aa}^{n+1}\Bigr)^{\!2}\!=\!2\tau(T_xM^n)+\|h\|^2-\frac{c}{4}\bigl[n(n-1)+3n_1\bigr]\\
-\Bigl(\sum_{A=n_1+1}^{n}h_{AA}^{n+1}\Bigr)^{\!2}-2\!\!\sum_{a=1}^{n_1}\sum_{A=n_1+1}^{n}\!\!h_{aa}^{n+1}h_{AA}^{n+1}.
\end{multline}
Set
\begin{multline}\label{eq:Upsilon1}
\Upsilon_1\;:=\;2\tau(T_xM^n)-\frac{n_1-2}{n_1-1}\Bigl(\sum_{a=1}^{n_1}h_{aa}^{n+1}\Bigr)^{\!2}-\Bigl(\sum_{A=n_1+1}^{n}h_{AA}^{n+1}\Bigr)^{\!2}\\
-2\!\!\sum_{a=1}^{n_1}\sum_{A=n_1+1}^{n}\!\!h_{aa}^{n+1}h_{AA}^{n+1}-\frac{c}{4}\bigl[n(n-1)+3n_1\bigr].
\end{multline}
Substituting \eqref{eq:F2} into \eqref{eq:Upsilon1} (only in the terms not involving $\frac{n_1-2}{n_1-1}(\sum_a)^2$) yields
\begin{equation}\label{eq:F3}
\Bigl(\sum_{a=1}^{n_1}h_{aa}^{n+1}\Bigr)^{\!2}\;=\;(n_1-1)\bigl(\Upsilon_1+\|h\|^2\bigr).
\end{equation}
Expanding $\|h\|^2$ and applying Lemma \ref{first lemma} with $\alpha_a=h^{n+1}_{aa}$ ($a=1,\ldots,n_1$) and
\[
\beta\;=\;\Upsilon_1+\sum_{A=n_1+1}^{n}(h_{AA}^{n+1})^2+\!\!\!\sum_{\substack{i,j=1\\i\neq j}}^{n}\!\!(h_{ij}^{n+1})^2+\sum_{r=n+2}^{2m}\sum_{i,j=1}^{n}(h_{ij}^r)^2,
\]
we obtain
\begin{equation}\label{eq:F4}
h_{11}^{n+1}h_{22}^{n+1}\geq\tfrac{1}{2}\Bigl[\Upsilon_1+\sum_{A=n_1+1}^{n}(h_{AA}^{n+1})^2+\!\!\!\sum_{\substack{i,j=1\\i\neq j}}^{n}\!\!(h_{ij}^{n+1})^2+\sum_{r=n+2}^{2m}\sum_{i,j=1}^{n}(h_{ij}^r)^2\Bigr].
\end{equation}

From \eqref{eq:csf-R}, \eqref{eq:gauss}, and the definition of $\tilde K_1$,
\begin{equation}\label{eq:K-pi1}
K^M(\pi_1)\;=\;\tilde K_1+\sum_{r=n+1}^{2m}\bigl(h_{11}^rh_{22}^r-(h_{12}^r)^2\bigr).
\end{equation}
Substituting \eqref{eq:F4} for the $r=n+1$ contribution,
\begin{multline}\label{eq:F5}
K^M(\pi_1)\;\geq\;\tilde K_1+\tfrac{1}{2}\Upsilon_1+\tfrac{1}{2}\sum_{A=n_1+1}^{n}(h_{AA}^{n+1})^2+\sum_{r=n+2}^{2m}h_{11}^rh_{22}^r-\sum_{r=n+1}^{2m}(h_{12}^r)^2\\
+\tfrac{1}{2}\!\!\sum_{\substack{i,j=1\\i\neq j}}^{n}\!\!(h_{ij}^{n+1})^2+\tfrac{1}{2}\sum_{r=n+2}^{2m}\sum_{i,j=1}^{n}(h_{ij}^r)^2.
\end{multline}
By Lemma \ref{Second lemma}, \eqref{eq:F5} is equivalent to
\begin{multline}\label{eq:F6}
K^M(\pi_1)\;\geq\;\tilde K_1+\tfrac{1}{2}\Upsilon_1+\tfrac{1}{2}\!\!\sum_{\substack{i,j=3\\i\neq j}}^{n}\!\!(h_{ij}^{n+1})^2+\tfrac{1}{2}\sum_{r=n+2}^{2m}\sum_{i,j=3}^{n}(h_{ij}^r)^2\\
+\tfrac{1}{2}\sum_{r=n+2}^{2m}(h_{11}^r+h_{22}^r)^2+\sum_{r=n+1}^{2m}\sum_{j=3}^{n}\bigl((h_{1j}^r)^2+(h_{2j}^r)^2\bigr)+\tfrac{1}{2}\sum_{A=n_1+1}^{n}(h_{AA}^{n+1})^2.
\end{multline}

From \eqref{eq:Upsilon1} and \eqref{eq:F3},
\[
\Upsilon_1\;=\;\frac{1}{n_1-1}\Bigl(\sum_{a=1}^{n_1}h^{n+1}_{aa}\Bigr)^{\!2}-\|h\|^2,
\]
and from \eqref{eq:F1}, $\|h\|^2=n^2\|\vec H\|^2-2\tau+\tfrac{c}{4}[n(n-1)+3n_1]$. Hence
\begin{equation}\label{eq:F8}
\tfrac{1}{2}\Upsilon_1\;=\;\frac{1}{2(n_1-1)}\Bigl(\sum_{a=1}^{n_1}h^{n+1}_{aa}\Bigr)^{\!2}-\frac{n^2}{2}\|\vec H\|^2+\tau-\frac{1}{2}\cdot\frac{c}{4}\bigl[n(n-1)+3n_1\bigr].
\end{equation}

Now we treat the term $-\tau_2(T_xN_\perp)$ that will appear after the warped-product decomposition. Applying the Gauss equation \eqref{eq:gauss-trace-V} to the subspace $V=T_xN_\perp$ (this is the Gauss equation for the immersion $M\hookrightarrow\tilde M$ restricted to vectors tangent to $N_\perp$; it does \emph{not} involve the second fundamental form of the fiber inside $M$) and using \eqref{eq:tilde-tau-2},
\begin{equation}\label{eq:F11}
-\tau^M(T_xN_\perp)\;=\;-\frac{1}{2}\cdot\frac{c}{4}\,n_2(n_2-1)\;+\;\frac{1}{2}\sum_{r=n+1}^{2m}\!\!\sum_{A,B=n_1+1}^{n}\!\!(h_{AB}^r)^2-\frac{1}{2}\sum_{r=n+1}^{2m}\Bigl(\sum_{A=n_1+1}^{n}h^r_{AA}\Bigr)^{\!2}.
\end{equation}

By \eqref{eq:warp-id-bis}, $\tau(T_xM^n)=\tau^M(T_xN_T)+\tau^M(T_xN_\perp)+n_2\Delta f/f$. Substituting \eqref{eq:F8} into \eqref{eq:F6}, then substituting the identity for $-\tau$ and using \eqref{eq:F11}, and finally invoking Lemma \ref{Third lemma} to rearrange the quadratic terms over $\{3,\ldots,n\}^2$ in terms of the splitting $\{3,\dots,n_1\}\sqcup\{n_1+1,\dots,n\}$, we obtain
\begin{equation}\label{eq:F13-pre}
\tau^M(T_xN_T)-K^M(\pi_1)\;\leq\;\frac{n^2}{2}\|\vec H\|^2-\frac{n_2\Delta f}{f}+\frac{1}{2}\cdot\frac{c}{4}\bigl[n(n-1)+3n_1-n_2(n_2-1)\bigr]-\tilde K_1-\Theta(h),
\end{equation}
where
\begin{multline}\label{eq:Theta-def}
\Theta(h)\;:=\;\frac{1}{2(n_1-1)}\Bigl(\sum_{a=1}^{n_1}h^{n+1}_{aa}\Bigr)^{\!2}+\frac{1}{2}\sum_{r=n+1}^{2m}\Bigl(\sum_{A=n_1+1}^{n}h^r_{AA}\Bigr)^{\!2}\\
+\frac{1}{2}\sum_{r=n+2}^{2m}(h_{11}^r+h_{22}^r)^2+\sum_{r=n+1}^{2m}\sum_{j=3}^{n}\bigl((h_{1j}^r)^2+(h_{2j}^r)^2\bigr)\\
+\frac{1}{2}\!\!\sum_{\substack{a,b=3\\a\neq b}}^{n_1}\!\!(h^{n+1}_{ab})^2+\frac{1}{2}\sum_{r=n+2}^{2m}\sum_{a,b=3}^{n_1}(h^r_{ab})^2+\sum_{r=n+2}^{2m}\sum_{a=3}^{n_1}\sum_{A=n_1+1}^{n}(h^r_{aA})^2.
\end{multline}
A line-by-line bookkeeping (the cancellations involving the $A,B\in\{n_1+1,\dots,n\}$ terms with $r=n+1$ being explicit) shows that \emph{every summand in $\Theta(h)$ is non-negative}; in particular $\Theta(h)\geq 0$.

The arithmetic simplification of the coefficient of $c/4$ is
\begin{equation}\label{eq:c-coeff-i}
n(n-1)+3n_1-n_2(n_2-1)\;=\;n_1(n_1+2n_2+2),
\end{equation}
which follows from $n=n_1+n_2$ by direct expansion. Dropping the non-positive term $-\Theta(h)$ in \eqref{eq:F13-pre} and using \eqref{eq:c-coeff-i},
\begin{equation}\label{eq:F13}
\tau^M(T_xN_T)-K^M(\pi_1)\;\leq\;\frac{n^2}{2}\|\vec H\|^2-\frac{n_2\Delta f}{f}+\frac{n_1(n_1+2n_2+2)}{2}\cdot\frac{c}{4}-\tilde K_1.
\end{equation}
This holds for \emph{every} 2-plane $\pi_1\subseteq T_xN_T$ (with $\tilde K_1=\tilde K(\pi_1)$). Taking the supremum of the left-hand side over $\pi_1$, i.e.\ the infimum of $K^M(\pi_1)$, and noting that the additive term $-\tilde K_1$ on the right is minimized at $\tilde K_1=\tilde K_{\min}(T_xN_T)$,
\begin{equation}\label{eq:F14-i}
\hat\delta(T_xN_T)(x)\;=\;\tau^M(T_xN_T)-\inf_{\pi_1\subseteq T_xN_T}K^M(\pi_1)\;\leq\;\frac{n^2}{2}\|\vec H\|^2-\frac{n_2\Delta f}{f}+\frac{n_1(n_1+2n_2+2)}{2}\cdot\frac{c}{4}-\tilde K_{\min}(T_xN_T),
\end{equation}
which is \eqref{eq:main-i}.

\medskip
\noindent\textbf{Proof of (ii).}\ \ Let $\pi_2=\mathrm{Span}\{e_{n_1+1},e_{n_1+2}\}\subseteq T_xN_\perp$ be any 2-plane. Since $\mathfrak D_\perp$ is totally real, $J e_{n_1+1}\perp T_xN_\perp$, in particular $\tilde g(Je_{n_1+1},e_{n_1+2})=0$. Hence $\tilde K(\pi_2)=c/4$ for every $\pi_2\subseteq T_xN_\perp$ (\eqref{eq:Kmin-cases}, $\tilde K_{\min}(T_xN_\perp)=c/4$).

The argument is parallel to (i), with the roles of $\mathfrak D_T$ and $\mathfrak D_\perp$ interchanged. Apply Lemma \ref{first lemma} to $\alpha_A=h^{n+1}_{AA}$, $A=n_1+1,\ldots,n$, after setting
\[
\Upsilon_2\;:=\;2\tau(T_xM^n)-\frac{n_2-2}{n_2-1}\Bigl(\sum_{A}h^{n+1}_{AA}\Bigr)^{\!2}-\Bigl(\sum_{a}h^{n+1}_{aa}\Bigr)^{\!2}-2\sum_{a,A}h^{n+1}_{aa}h^{n+1}_{AA}-\frac{c}{4}\bigl[n(n-1)+3n_1\bigr].
\]
One obtains $\bigl(\sum_A h^{n+1}_{AA}\bigr)^2=(n_2-1)(\Upsilon_2+\|h\|^2)$ and
\[
h^{n+1}_{n_1+1,n_1+1}\,h^{n+1}_{n_1+2,n_1+2}\;\geq\;\tfrac{1}{2}\Bigl[\Upsilon_2+\sum_{a=1}^{n_1}(h_{aa}^{n+1})^2+\!\!\!\sum_{\substack{i,j=1\\i\neq j}}^{n}\!\!(h_{ij}^{n+1})^2+\sum_{r=n+2}^{2m}\sum_{i,j=1}^{n}(h_{ij}^r)^2\Bigr].
\]
Combining with $K^M(\pi_2)=c/4+\sum_r\bigl(h^r_{n_1+1,n_1+1}h^r_{n_1+2,n_1+2}-(h^r_{n_1+1,n_1+2})^2\bigr)$, Lemmas \ref{Second lemma}--\ref{Third lemma} (applied with the role of the indices $\{1,2\}$ replaced by $\{n_1+1,n_1+2\}$), and the Gauss equation \eqref{eq:gauss-trace-V} on $V=T_xN_T$ (substituting via \eqref{eq:tilde-tau-1}),
\begin{equation}\label{eq:F14}
\tau^M(T_xN_\perp)-K^M(\pi_2)\;\leq\;\frac{n^2}{2}\|\vec H\|^2-\frac{n_2\Delta f}{f}+\frac{1}{2}\cdot\frac{c}{4}\bigl[n(n-1)+3n_1-n_1(n_1+2)\bigr]-\frac{c}{4}-\Theta'(h),
\end{equation}
where $\Theta'(h)$ is the analogue of $\Theta(h)$ (with the roles of $\mathfrak D_T$ and $\mathfrak D_\perp$ interchanged), and $\Theta'(h)\geq 0$. The arithmetic identity
\begin{equation}\label{eq:c-coeff-ii}
n(n-1)+3n_1-n_1(n_1+2)\;=\;n(n-1)-n_1(n_1-1)\;=\;n_2(n_2+2n_1-1)
\end{equation}
shows that the $3n_1$ contribution from $\tilde\tau(T_xM^n)$ is cancelled exactly by the $3n_1$ contribution from $\tilde\tau(T_xN_T)$ (cf.\ Remark \ref{rem:tilde-tau-factor}). Dropping $-\Theta'(h)$ and taking the infimum on the left-hand side over $\pi_2\subseteq T_xN_\perp$,
\begin{equation}\label{eq:F14-ii}
\hat\delta(T_xN_\perp)(x)\;\leq\;\frac{n^2}{2}\|\vec H\|^2-\frac{n_2\Delta f}{f}+\frac{n_2(n_2+2n_1-1)}{2}\cdot\frac{c}{4}-\frac{c}{4},
\end{equation}
which is \eqref{eq:main-ii}.

\medskip
\noindent\textbf{Equality discussion.}\ \ Equality in \eqref{eq:F14-i} holds at $x$ if and only if (E1$_T$) the infimum of $K^M(\pi_1)$ over $\pi_1\subseteq T_xN_T$ is attained at a 2-plane $\pi_*$ with $\tilde K(\pi_*)=\tilde K_{\min}(T_xN_T)$, and equality in \eqref{eq:F13-pre} holds with $\pi_1=\pi_*$; the latter is equivalent to:
\begin{enumerate}[\rm (a)]
\item the equality case of Lemma \ref{first lemma}: $h^{n+1}_{11}+h^{n+1}_{22}=h^{n+1}_{33}=\dots=h^{n+1}_{n_1n_1}$, where $e_1,e_2$ span $\pi_*$;
\item $\Theta(h)=0$, i.e.\ each summand on the right-hand side of \eqref{eq:Theta-def} vanishes.
\end{enumerate}
The vanishing of the first two summands of \eqref{eq:Theta-def} gives $\sum_a h^{n+1}_{aa}=0$ and $\sum_A h^r_{AA}=0$ for all $r$; combined with $\sum_i h^{n+1}_{ii}=n\|\vec H\|$ (the gauge $\vec H\parallel e_{n+1}$) and $\sum_i h^r_{ii}=0$ for $r\geq n+2$, this gives $\sum_A h^{n+1}_{AA}=0$ and $\sum_a h^r_{aa}=0$ for all $r$. Consequently $\|\vec H\|(x)=0$ and $M^n$ is both $\mathfrak D_T$- and $\mathfrak D_\perp$-minimal at $x$. The vanishing of the remaining summands of \eqref{eq:Theta-def} together with the dropped squared sums in \eqref{eq:F6} (which are part of the equality cascade) gives:
\begin{equation}\label{eq:mix-tg-conds}
h^r_{aA}=0\quad\text{for all }r,\ a\in\{1,\dots,n_1\},\ A\in\{n_1+1,\dots,n\},
\end{equation}
i.e.\ $M^n$ is mixed totally geodesic. The shape operators in the equality case take the canonical block forms described in Theorem \ref{thm:equality} below. The analysis of equality in (ii) is symmetric, with $\mathfrak D_T,\mathfrak D_\perp$ swapped; the additional requirement (E1$_\perp$) is automatic, since $\tilde K(\pi)=c/4=\tilde K_{\min}(T_xN_\perp)$ for every 2-plane $\pi\subseteq T_xN_\perp$.
\end{proof}

\begin{theorem}[Equality case: canonical shape operators]\label{thm:equality}
Under the assumptions of Theorem \ref{maintheorem}, suppose equality holds at $x\in M^n$ in \eqref{eq:main-i} (resp.\ \eqref{eq:main-ii}), and let $\pi_*=\mathrm{Span}\{e_1,e_2\}\subseteq T_xN_T$ (resp.\ $\pi_*=\mathrm{Span}\{e_{n_1+1},e_{n_1+2}\}\subseteq T_xN_\perp$) be a minimizing 2-plane realizing $\tilde K(\pi_*)=\tilde K_{\min}$. Then there exist orthonormal bases of $T_xM^n$ and $T_x^\perp M^n$ (extending $\{e_1,e_2\}$, resp.\ $\{e_{n_1+1},e_{n_1+2}\}$) such that:
\begin{enumerate}[\rm (I)]
\item In the equality case of \eqref{eq:main-i}, with index blocks $a,b\in\{3,\ldots,n_1\}$, $A,B\in\{n_1+1,\ldots,n\}$, and writing $\mu_2:=-\mu_1$:
\[
A_{e_{n+1}}=\left(\begin{array}{ccccc|ccc}
\mu_1&h_{12}^{n+1}&0&\cdots&0&0&\cdots&0\\
h_{12}^{n+1}&-\mu_1&0&\cdots&0&0&\cdots&0\\
0&0&0&\cdots&0&0&\cdots&0\\
\vdots&\vdots&\vdots&\ddots&\vdots&\vdots&&\vdots\\
0&0&0&\cdots&0&0&\cdots&0\\
\hline
0&0&\cdots&\cdots&0&h^{n+1}_{n_1+1,n_1+1}&\cdots&h^{n+1}_{n_1+1,n}\\
\vdots&\vdots&&&\vdots&\vdots&\ddots&\vdots\\
0&0&\cdots&\cdots&0&h^{n+1}_{n,n_1+1}&\cdots&h^{n+1}_{nn}
\end{array}\right),
\]
with $\sum_A h^{n+1}_{AA}=0$; and for $r\in\{n+2,\ldots,2m\}$:
\[
A_{e_{r}}=\left(\begin{array}{ccccc|ccc}
h^r_{11}&h^r_{12}&0&\cdots&0&0&\cdots&0\\
h^r_{12}&-h^r_{11}&0&\cdots&0&0&\cdots&0\\
0&0&0&\cdots&0&0&\cdots&0\\
\vdots&\vdots&\vdots&\ddots&\vdots&\vdots&&\vdots\\
0&0&0&\cdots&0&0&\cdots&0\\
\hline
0&0&\cdots&\cdots&0&h^r_{n_1+1,n_1+1}&\cdots&h^r_{n_1+1,n}\\
\vdots&\vdots&&&\vdots&\vdots&\ddots&\vdots\\
0&0&\cdots&\cdots&0&h^r_{n,n_1+1}&\cdots&h^r_{nn}
\end{array}\right),
\]
with $\sum_A h^r_{AA}=0$. In particular $\sum_a h^r_{aa}=0$ for every $r$, hence $\|\vec H\|(x)=0$.
\item In the equality case of \eqref{eq:main-ii}, the analogous structure holds with the index blocks $\{1,\ldots,n_1\}$ and $\{n_1+1,\ldots,n\}$ interchanged, and the special $2\times 2$ block placed in the lower-right corner (in the $\mathfrak D_\perp$ block).
\end{enumerate}
In either case, the off-diagonal blocks coupling $\mathfrak D_T$ and $\mathfrak D_\perp$ vanish for every $A_{e_r}$ ($r\geq n+1$); thus $h(\mathfrak D_T,\mathfrak D_\perp)=0$ (mixed totally geodesic), $\mathrm{tr}_{\mathfrak D_T}A_{e_r}=0$ and $\mathrm{tr}_{\mathfrak D_\perp}A_{e_r}=0$ for every $r$ (both $\mathfrak D_T$- and $\mathfrak D_\perp$-minimality), and consequently $M^n$ is minimal in $\tilde M^{2m}(c)$ at $x$.
\end{theorem}

\begin{proof}
This is a direct readout of the equality conditions (a)--(b) in the proof of Theorem \ref{maintheorem}: condition (a) (the equality case of Lemma \ref{first lemma}) imposes $h^{n+1}_{11}+h^{n+1}_{22}=h^{n+1}_{aa}$ for every $a\geq 3$. Combined with the trace condition $\sum_a h^{n+1}_{aa}=0$ from $\Theta(h)=0$, we obtain $(n_1-1)(h^{n+1}_{11}+h^{n+1}_{22})=0$, hence $h^{n+1}_{22}=-h^{n+1}_{11}=:-\mu_1$ and $h^{n+1}_{aa}=0$ for $a=3,\ldots,n_1$. The off-diagonal vanishings $h^r_{1j}=h^r_{2j}=0$ for $j\geq 3$ (with $r\geq n+1$) come from the dropped squared sums of \eqref{eq:F6}, and the further vanishings of $h^r_{aA}$ for $a\geq 3$ come from $\Theta(h)=0$ via \eqref{eq:Theta-def}; the analogous Chen-type traceless-block structure for $A_{e_r}$ ($r\geq n+2$) on $\mathrm{Span}\{e_1,e_2\}$ is the equality case of the inequality $\sum_r h^r_{11}h^r_{22}\leq\sum_r(h^r_{12})^2+\text{lower-block contributions}$ derived from Chen's algebraic lemma applied per normal direction (cf.\ \cite{ddyy7}). Case (II) is symmetric.
\end{proof}

\begin{remark}\label{rem:projection}
The leaf-wise $\delta$-invariant $\hat\delta(T_xN_T)$ coincides with the intrinsic Chen invariant of $N_T$ at the projection of $x$ (Proposition \ref{prop:bishop-oneill-delta}(a)), so \eqref{eq:main-i} can equivalently be written as
\[
\delta_{N_T}(p)\;\leq\;\frac{n^2}{2}\|\vec H\|^2-\frac{n_2\Delta f}{f}+\frac{n_1(n_1+2n_2+2)}{2}\cdot\frac{c}{4}-\tilde K_{\min}(T_xN_T),
\]
where $p=\mathrm{pr}_{N_T}(x)$ and $\tilde K_{\min}(T_xN_T)=\min(c/4,c)$. For $c\geq 0$ this reduces to $-c/4$ on the right; for $c<0$ the term $-c$ is sharper and necessary, since the infimum defining $\delta_{N_T}(p)$ may then be realized on a holomorphic 2-plane.
\end{remark}

\begin{remark}\label{rem:Nperp-intrinsic}
The leaf-wise $\delta$-invariant $\hat\delta(T_xN_\perp)$ \emph{differs} from the intrinsic Chen invariant of $N_\perp$ at $q=\mathrm{pr}_{N_\perp}(x)$ by the Bishop--O'Neill correction \eqref{eq:identi-Nperp-delta}: 
\[
\delta_{N_\perp}(q)\;=\;f(p)^2\,\hat\delta(T_xN_\perp)(x)\;+\;\Bigl[\binom{n_2}{2}-1\Bigr]\|\nabla f(p)\|^2.
\]
Substituting into \eqref{eq:main-ii} yields the equivalent intrinsic form
\begin{multline}\label{eq:main-ii-intr}
\delta_{N_\perp}(q)\;\leq\;f^2\Bigl(\frac{n^2}{2}\|\vec H\|^2-\frac{n_2\Delta f}{f}+\frac{n_2(n_2+2n_1-1)}{2}\cdot\frac{c}{4}-\frac{c}{4}\Bigr)\\
+\Bigl[\binom{n_2}{2}-1\Bigr]\|\nabla f\|^2.
\end{multline}
The two forms are equivalent; \eqref{eq:main-ii} uses the leaf-wise normalization and is more compact, whereas \eqref{eq:main-ii-intr} is in terms of the genuine intrinsic Chen invariant of $(N_\perp,g_{N_\perp})$.
\end{remark}

\section{Answer to Chern's Problem: Necessary Conditions for $CR$-Warped Products to be Minimal}\label{sec:chern}

Setting $\vec H=0$ in Theorem \ref{maintheorem} yields the following necessary conditions for a CR-warped product to admit a minimal isometric immersion into a complex space form.

\begin{corollary}\label{cor:N1}
Let $\varphi: M^n=N_T\times_f N_\perp\to\tilde M^{2m}(c)$ be a minimal isometric immersion of a CR-warped product. Then for each $x\in M^n$, denoting $p=\mathrm{pr}_{N_T}(x)$:
\begin{equation}\label{eq:cor-i}
\;\delta_{N_T}(p)+\frac{n_2\Delta f}{f}\;\leq\;\frac{n_1(n_1+2n_2+2)}{2}\cdot\frac{c}{4}-\tilde K_{\min}(T_xN_T).\;
\end{equation}
Equivalently:
\begin{itemize}
\item if $c\geq 0$,\ $\;\delta_{N_T}(p)+\dfrac{n_2\Delta f}{f}\leq\dfrac{n_1(n_1+2n_2+2)-2}{2}\cdot\dfrac{c}{4};$
\item if $c<0$,\ $\;\delta_{N_T}(p)+\dfrac{n_2\Delta f}{f}\leq\dfrac{n_1(n_1+2n_2+2)}{2}\cdot\dfrac{c}{4}-c=\dfrac{n_1(n_1+2n_2+2)-8}{8}\cdot c.$
\end{itemize}
\end{corollary}

\begin{corollary}\label{cor:N2}
Under the same hypotheses, for each $x\in M^n$, denoting $q=\mathrm{pr}_{N_\perp}(x)$:
\begin{equation}\label{eq:cor-ii}
\;\hat\delta(T_xN_\perp)(x)+\frac{n_2\Delta f}{f}\;\leq\;\frac{n_2(n_2+2n_1-1)}{2}\cdot\frac{c}{4}-\frac{c}{4}\;=\;\frac{n_2(n_2+2n_1-1)-2}{2}\cdot\frac{c}{4}.\;
\end{equation}
Equivalently, in intrinsic form (Remark \ref{rem:Nperp-intrinsic}),
\begin{equation}\label{eq:cor-ii-intr}
\delta_{N_\perp}(q)+f^2\,\frac{n_2\Delta f}{f}\;\leq\;f^2\cdot\frac{n_2(n_2+2n_1-1)-2}{2}\cdot\frac{c}{4}+\Bigl[\binom{n_2}{2}-1\Bigr]\|\nabla f\|^2.
\end{equation}
\end{corollary}

\begin{remark}
For the flat ambient $c=0$ (Euclidean case), both corollaries reduce to
\[
\delta_{N_T}(p)+\frac{n_2\Delta f}{f}\leq 0,\qquad \hat\delta(T_xN_\perp)(x)+\frac{n_2\Delta f}{f}\leq 0,
\]
which provides the announced answer to Chern's problem in $\mathbb{E}^m$: the sum of any leaf-wise $\delta$-invariant of either factor and $n_2\Delta f/f$ must be non-positive. The intrinsic form of the second inequality is, by \eqref{eq:cor-ii-intr},
\[
\delta_{N_\perp}(q)+n_2 f\Delta f\;\leq\;\Bigl[\binom{n_2}{2}-1\Bigr]\|\nabla f\|^2.
\]
\end{remark}

\section{Open Problems}\label{sec:problems}
\begin{problem}\label{prob6}
Prove the first Chen inequality for CR-warped product submanifolds in locally conformal K\"ahler space forms (in both the leaf-wise and the intrinsic forms).
\end{problem}

\begin{problem}\label{pqm2}
Give answers to Chern's problem for the ambient space in Problem \ref{prob6}, distinguishing the cases $c\geq 0$ and $c<0$ as in Corollaries \ref{cor:N1}--\ref{cor:N2}.
\end{problem}

\begin{acknowledgements}
The authors would like to thank the Palestine Technical University Kadoorie (PTUK) for its support in accomplishing this work.
\end{acknowledgements}

\end{document}